\documentclass[11pt]{amsart}
\usepackage{amssymb}
\usepackage[T1]{fontenc}

\usepackage{anysize}
\marginsize{2.5cm}{2.5cm}{2.5cm}{2.5cm}
\newcommand{\Z}{\mathbb{Z}}
\newcommand{\N}{\mathbb{N}}
\newcommand{\SL}{{\text {\rm SL}}}

\newcommand{\Q}{\mathbb{Q}}
\newcommand{\Gal}{\textrm{Gal}(\overline{\Q}/\Q)}
\newcommand{\trace}{\textrm{trace}}
\newcommand{\Frob}{\textrm{Frob}}
\newcommand{\GL}{\mathbb{GL}}

\theoremstyle{plain}
\newtheorem{theorem}{Theorem}
\newtheorem{proposition}{Proposition}

\theoremstyle{remark}
\newtheorem{remark}{Remark}

\title{Congruences for sporadic sequences and modular forms for non-congruence subgroups}

\author{Matija Kazalicki}
\address{Department of Mathematics, University of Zagreb, Bijeni\v{c}ka cesta 30, 10000 Zagreb, Croatia}
\email{matija.kazalicki@math.hr}


\begin{document}

\begin{abstract}

In 1979, in the course of the proof of the irrationality of $\zeta(2)$  Robert Ap\'ery introduced numbers $b_n = \sum_{k=0}^n {n \choose k}^2{n+k \choose k}$ that are, surprisingly, integral solutions of recursive relations 
$$(n+1)^2 u_{n+1} - (11n^2+11n+3)u_n-n^2u_{n-1} = 0.$$
Zagier performed a computer search on first 100 million triples $(A,B,C)\in \Z^3$ and found that the recursive relation generalizing $b_n$
$$(n+1)u_{n+1} - (An^2+An+B)u_n + C n ^2 u_{n-1}=0,$$
with the initial conditions $u_{-1}=0$ and $u_0=1$ has (non-degenerate i.e. $C(A^2-4C)\ne 0$) integral solution for only six more triples (whose solutions are so called sporadic sequences) .

Stienstra and Beukers showed that for the prime $p\ge 5$ 
\begin{equation*}
b_{(p-1)/2} \equiv \begin{cases} 4a^2-2p \pmod{p} \textrm{ if } p = a^2+b^2,\textrm{ a odd}\\ 0 \pmod{p} \textrm{ if } p\equiv 3 \pmod{4}.\end{cases}
\end{equation*}

Recently, Osburn and Straub proved similar congruences for all but one of the six Zagier's sporadic sequences (three cases were already known to be true by the work of Stienstra and Beukers) and conjectured the congruence for the sixth sequence (which is a solution of recursion determined by triple $(17,6,72)$. 

In this paper we prove that remaining congruence by studying Atkin and Swinnerton-Dyer congruences between Fourier coefficients of certain cusp form for non-congurence subgroup.

\end{abstract}

\maketitle

\section{Introduction}

\noindent In 1979, in the course of his famous proof of the irrationality of $\zeta(3)$ and $\zeta(2)$ Robert Ap\'ery \cite{Ap} introduced numbers $a_n = \sum_{k=0}^n {n \choose k}^2{n+k \choose k}^2$ and $b_n = \sum_{k=0}^n {n \choose k}^2{n+k \choose k}$. These numbers, which was important for the proof, are integral solutions of recursive relations 
$$(n+1)^3 u_{n+1}-(34n^3+51n^2+27n+5)u_n + n^3 u_{n-1}=0 \quad \textrm{ and }$$
$$(n+1)^2 u_{n+1} - (11n^2+11n+3)u_n-n^2u_{n-1} = 0$$
respectively. The integrality came as a big surprise since to calculate $a_n$ (or $b_n$) in each step one has to divide by $n^3$ (or $n^2$) so a priori one would expect that these numbers have denominators of the size $n!^3$ (or $n!^2)$. Inspired by Beukers \cite{B1}, Zagier \cite{Zag1} performed a computer search on first 100 million triples $(A,B,C)\in \Z^3$ and found that the recursive relation generalizing $b_n$
$$(n+1)u_{n+1} - (An^2+An+B)u_n + C n ^2 u_{n-1}=0,$$
with the initial conditions $u_{-1}=0$ and $u_0=1$ has (non-degenerate i.e. $C(A^2-4C)\ne 0$) integral solution for only six more triples (whose solutions are so called sporadic sequences) 
$$(0,0,-16), (7,2,-8), (9,3,27), (10,3,9), (12,4,32) \textrm{ and } (17,6,72).$$

Interestingly, Stienstra and Beukers \cite{SB} showed that the generating function of Ap\'ery's numbers $b_n$ is a holomorphic solution of Picard-Fuchs differential equation of elliptic K3-surface $\mathcal{S}:X(Y-Z)(Z-X)-t(X-Y)YZ=0$ (other sporadic sequences are related in this way to K3 surfaces as well, see \cite{Zag1}). Using this connection they also proved that for prime $p\ge 5$ 
\begin{equation*}
b_{(p-1)/2} \equiv \begin{cases} 4a^2-2p \pmod{p} \textrm{ if } p = a^2+b^2,\textrm{ a odd}\\ 0 \pmod{p} \textrm{ if } p\equiv 3 \pmod{4}.\end{cases}
\end{equation*}
Here one can interpret the right-hand side of the congruences as a $p$-th Fourier coefficient of a certain $CM$ modular form of weight 3 whose $L$-function is a factor of the zeta function of $\mathcal{S}$. (Later Beukers \cite{B1} proved a similar result for the numbers $a_n$ - this time relating them to the coefficients of Hecke eigenform of weight 4.) For a beautiful survey of these results see \cite{Zag2}.

Recently, Osburn and Straub \cite{OS} proved similar congruences for all but one of the six Zagier's sporadic sequences (three cases were already known to be true by the work of Stienstra and Beukers) and conjectured the congruence for the sixth sequence $F(n)$ (which is a solution of recursion determined by triple $(17,6,72)$. In this paper we prove that remaining congruence.

Denote by $$F(n) = \sum_{k=0}^n (-1)^k 8^{n-k}{n \choose k}\sum_{j=0}^k {k \choose j}^3,$$ the sporadic sequence corresponding to triple $(17,6,72)$. For $\tau \in \mathbb{H}$ and $q=e^{2\pi i \tau}$ let $$f(\tau)=\sum_{n=0}^\infty=q-2q^2+3q^3+\cdots=\sum_{n=0}^\infty \gamma(n)q^n \in S_3\left(\Gamma_0(24),\left( \frac{-6}{\cdot}\right)\right)$$ be a newform. Our main result is the following theorem.

\begin{theorem}\label{thm:1}
For all primes $p>2$ we have
$$F\left(\frac{p-1}{2}\right) \equiv \gamma(p) \pmod{p}.$$
\end{theorem}
\begin{remark}
One can check that $f(\tau)$ is CM form such that for prime $p$
\begin{equation*}
\gamma(p) \equiv \begin{cases} 2(a^2-6b^2) \textrm{ if } p = a^2+6b^2\\ 0 \pmod{p} \textrm{ if } p \equiv 5, 11, 13, 17, 19, 23 \pmod{24}.\end{cases}
\end{equation*}

\end{remark}
In Section \ref{sec:2} using the method of Beukers \cite[Proposition 3.]{B1} and Verrill \cite[Theorem 1.1]{V} we reduce Theorem \ref{thm:1} to showing that the weight three cusp form (for non-congruence subgroup $\Gamma_2$ of $\Gamma_1(6)$)  
$$g(\tau) =q^{1/2} + \frac{3}{2} q^{3/2} - \frac{9}{8} q^{5/2} - \frac{85}{16} q^{7/2} - \frac{981}{128} q^{9/2}+\cdots \in S_3(\Gamma_2),$$  
satisfies a three-term Atkin and Swinnerton-Dyer congruence relation with respect to $f(\tau)$ for all primes $p>3$ (see Proposition \ref{prop:2}).
The similar idea was used previously by the author \cite{K} in proving three term congruence relations for some multinomial sums by employing Atkin and Swinnerton-Dyer congruence relations satisfied by the Fourier coefficients of certain weakly holomorphic modular forms (but for congruence subgroups).

In Section \ref{sec:ASD} we explain how using Scholl's theory \cite{Scholl} we can reduce Proposition \ref{prop:2} to the equivalence of two strictly compatible families of $\ell$-adic Galois representations: $\tilde{\rho}_\ell$ isomorphic to $\ell$-adic realization of the motive associated to the space of cusp forms $S_3(\Gamma_2)$, and $\rho'_\ell$ attached to the newform $f(\tau)\otimes \left( \frac{-1}{\cdot}\right)$ by Deligne's work.

In Section \ref{sec:4} and Section \ref{sec:Serre} we prove that these two $\ell$-adic families are isomorphic by showing that they are isomorphic to the third $\ell$-adic family $\rho_\ell$ which is constructed from the explicit model of the universal family of elliptic curves over modular curve of $\Gamma_2$.

\section{Elliptic surfaces, modular forms and the proof of Theorem \ref{thm:1}}\label{sec:2}

Consider modular rational elliptic surface attached to $\Gamma_1(6)$ (see third example in \cite[Section 4.2.2.]{V})
$$\mathcal{W}: (x+y)(x+z)(y+z) - 8xyz =\frac{1}{t} xyz,$$
with fibration $\phi:\mathcal{W} \rightarrow P^1$, $(x,y,z,t) \mapsto t$. For $t \notin \{\infty, 0, -\frac{1}{9}, -\frac{1}{8} \}$ the preimage $\phi^{-1}(t)$ is an elliptic curve with a distinguished point of order $6$. Picard-Fuchs differential equation associated to this elliptic surface 
$$(8t + 1)(9t + 1)P(t)''+ t(144t + 17)P(t)' + 6t(12t + 1)P(t) = 0,$$
has a holomorphic solution around $t=0$ 
$$P(t) = \sum_{n=0}^\infty (-1)^n F(n)t^n.$$
(Our notation is slightly different from \cite[Section 4.2.2.]{V} since $F(n)=(-1)^n c_n$, with $c_n$ defined in \cite{V}) If we identify $t$ with a modular function (for $\Gamma_0(6)$)
$$t(\tau) = \frac{\eta(2\tau)\eta(6\tau)^5}{\eta(\tau)^5\eta(3\tau)}, \quad \tau \in \mathbb{H}$$
then $P(\tau) := \sum_{n=0}^\infty (-1)^n F(n) t(\tau)^n$ is a weight one modular form for $\Gamma_1(6)$.

Now consider a two cover $\mathcal{S}$ of $\mathcal{W}$, a K3-surface given by the equation 
$$\mathcal{S}: (x+y)(x+z)(y+z) - 8xyz = \frac{1}{s^2}xyz,$$
where $t = s^2$. Then $s(\tau)=\sqrt{\frac{\eta(2\tau)\eta(6\tau)^5}{\eta(\tau)^5\eta(3\tau)}}$ is a corresponding modular function for index two genus zero subgroup $\Gamma_2 \subset \Gamma_1(6)$. 

By identifying $s$-line with the modular curve $X(\Gamma_2)$, we can identify singular fibers of K3-surface $\mathcal{S}$ with cusps of modular curve $X(\Gamma_2)$. More precisely, using Tate's algorithm one finds that Kodaira types of singular fibers at $s=\infty, 0, \pm \frac{i}{2\sqrt{2}}$ and $\pm\frac{i}{3}$ are $I_2, I_{12}, I_3, I_3, I_2$ and $I_2$ respectively. Hence all the cusps of $X(\Gamma_2)$ are regular.

In general, for a finite index subgroup $\Gamma$ of $\SL_2(\Z)$ of genus $g$ such that $-I\notin \Gamma$ and $k$ odd, \cite[Theorem 2.25]{Shi} gives the formula for the dimension of $S_k(\Gamma)$
$$\dim S_k(\Gamma) = (k-1)(g-1)+\frac{1}{2}(k-2)r_1+\frac{1}{2}(k-1)r_2+\sum_{i=1}^{j}\frac{e_i-1}{2e_i},$$
where $r_1$ is the number of regular cusps, $r_2$ is the number of irregular cusps, and $e_i's$ are the orders of elliptic points. Since $\Gamma_2$ has no elliptic points ($\Gamma_1(6)$ is a free group), we have that $\dim S_3(\Gamma_2)=1$.

Our starting point for studying congruences involving $F(n)$ is the following proposition of Beukers \cite{B1}.

\begin{proposition}[Beukers]\label{prop:Beukers}
Let $p$ be a prime and 
$$\omega(t)=\sum_{n=1}^\infty b_nt^{n-1}dt$$
a differential form with $b_n\in \Z_p$. Let $t(q)=\sum_{n=1}^\infty A_nq^n$,$A_n\in \Z_p$, and suppose
$$\omega(t(q))=\sum_{n=1}^\infty c_n q^{n-1}dq.$$
Suppose there exist $\alpha_p,\beta_p\in \Z_p$ with $p|\beta_p$ such that 
$$b_{mp^r}-\alpha_p b_{mp^{r-1}}+\beta_pb_{mp^{r-2}}\equiv 0 \pmod{p^r}, \quad \forall m,r\in \N.$$
Then
$$c_{mp^r}-\alpha_p c_{mp^{r-1}}+\beta_pc_{mp^{r-2}}\equiv 0 \pmod{p^r}, \quad \forall m,r\in \N.$$
Moreover, if $A_1$ is $p$-adic unit then the second congruence implies the first, and we have that $b_p \equiv \alpha_p b_1 \pmod{p}$.
\end{proposition}

\noindent Given prime $p>2$, if we apply the previous proposition to a differential form $$\omega(s) = \sum_{n=0}^\infty (-1)^nF(n)s^{2n}ds,$$
and $s(q)$ - the $q$-expansion of modular function $s(\tau)$ (where $q=e^{\pi i \tau}$), we obtain that $\omega(s(q))=\sum_{n=1}^\infty c_n q^{n-1}dq$, where $c_n$ are Fourier coefficients of weight $3$ cusp form $g(\tau)\in S_3(\Gamma_2)$
$$g(q) = P(q)q\frac{d}{dq}s(q)=q + \frac{3}{2} q^3 - \frac{9}{8} q^5 - \frac{85}{16} q^7 - \frac{981}{128} q^9+\cdots=\sum_{n=1}^\infty c_n q^n.$$ 
\begin{remark}
\begin{itemize}
\item[a)]  For $p=2$ the Fourier coefficients of $s(q)$ are not $p$-integral so we can not use Proposition \ref{prop:Beukers}.
\item[b)] It is well known that a differential operator $q\frac{d}{dq}$ maps modular functions to meromorphic modular forms of weight $2$. Holomorphicity and cuspidality of $g(\tau)$ then follow since zeros of $P(\tau)$ cancel out the poles of $s(\tau)$.
\item[c)] Since Fourier coefficients of $g(\tau)$ have unbounded denominators, it follows that $\Gamma_2$ is non-congruence subgroup of $\Gamma_1(6)$ (for congruence subgroups the Hecke eigenforms (which form the basis for the space of cuspforms) have Fourier coefficients that are algebraic integers).
\end{itemize}
\end{remark}

We will show that, for all primes $p>3$, the cusp form $g(\tau)$ satisfies a three term  Atkin and Swinnerton-Dyer congruence relation with respect to the quadratic twist of the newform $f(\tau)=\sum_{n=1}^\infty \gamma(n) e^{2\pi i \tau}$ by quadratic character $\left(\frac{-1}{\cdot} \right)$.
Hence Theorem \ref{thm:1} follows from Proposition \ref{prop:Beukers} and the following proposition.

\begin{proposition}\label{prop:2}
Let $p>3$ be a prime. Then for all $m,r \in \N$, we have that
$$c_{mp^r}-\left(\frac{-1}{p} \right)\gamma(p) c_{mp^{r-1}}+\left( \frac{-6}{p}\right)p^2 c_{mp^{r-2}}\equiv 0 \pmod{p^{2r}}.$$
\end{proposition}

\section{Atkin and Swinnerton-Dyer congruences for $S_3(\Gamma_2)$}\label{sec:ASD}

For a finite index non-congruence subgroup $\Gamma \subset \SL_2(\Z)$ and a prime $p$, we say that weight $k$ cusp form $f(\tau)=\sum_{n=0}^\infty a_n q^n \in S_k(\Gamma, \overline{\Z_p})$ satisfy Atkin and Swinnerton-Dyer (ASD) congruence at $p$ if there exist an algebraic integer $A_p$ and a root of unity $\mu_p$ such that for all non-negative integers $m$ and $r$ we have
\begin{equation}\label{eq:1}
a_{mp^r}-A_p a_{mp^{r-1}} + \mu_p p^{k-1}a_{mp^{r-2}}\equiv 0 \pmod{p^{(k-1)r}}.
\end{equation}
(In our example ${a_{n}}'s$ and ${A_{p}}'s$ are rational integers, and $\mu_p=\pm 1$.)

In the absence of the useful theory of Hecke operators for non-congruence subgroups, such $f(\tau)$ can be regarded as Hecke eigenfunction at prime $p$.
A discovery of these congruences by Atkin and Swinnerton-Dyer \cite{ASD} initiated a systematic study of modular forms for non-congruence subgroups. For more information see a survey article by Li, Long and Yang\cite{LL}.

In the case when the space of cusp forms is one dimensional and generated by $f(\tau)$ (which is the case for $S_3(\Gamma_2)$ and $g(\tau)$), Scholl \cite{Scholl} proved that the ASD congruence holds for all but finitely many $p$.
The congruences were obtained by embedding the module of cusp forms into certain de Rham cohomology group $DR(\Gamma,k)$ which is the de Rham realization of the motive associated to the relevant space of modular forms. At a good prime $p$, crystalline theory endows $DR(\Gamma,k)\otimes \Z_p$ with a Frobenius endomorphism whose action on $q$-expansion gives rise to Atkin and Swinnerton-Dyer congruences, i.e. if $T^2-A_p T+\mu_p p^2$ is a characteristic polynomial of Frobenius acting on $DR(\Gamma,k)\otimes \Z_p$ then congruence \eqref{eq:1} holds ($A_p$ is the trace of Frobenius).
See \cite[Section 2]{KS} for the summary of these results. 

To calculate the trace of Frobenius $A_p$, following Scholl \cite[Sections 4 and 5]{Scholl}, we associate to the subgroup $\Gamma_2$ a strictly compatible family of $\ell$-adic Galois representations of $\Gal$, $\tilde{\rho}_\ell$, that is isomorphic to $\ell$-adic realization of the motive associated to the space of cusp forms $S_3(\Gamma_2)$. From \cite[2.7. Proposition]{Scholl2} and algebraic relation between $s(\tau)$ and modular $j$-invariant $j(\tau)$
$$(s^2-\frac{1}{6})^3(s^6-\frac{7}{2}s^4+\frac{3}{4}s^2-\frac{1}{24})^3+\frac{1}{72}(s-\frac{1}{3})^2(s+\frac{1}{3})^2 s^{12} (s^2-\frac{1}{8})^3 j = 0,$$
it follows that $\tilde\rho_\ell$ is unramified outside $2, 3$ and $\ell$.

In particular, for $\ell=2$ and prime $p>3$ we have that \cite[Theorem 5.4.]{Scholl}
\begin{equation}\label{eq:2}
A_p=trace(\tilde\rho_2(\Frob_p)) \textrm{  and  } \mu_p=\det(\tilde\rho_2(Frob_p)).
\end{equation}

\section{Compatible families of $\ell$-adic Galois representations of $\Gal$}\label{sec:4}

Denote by $\rho_\ell'$ a strictly compatible family of two dimensional $\ell$-adic Galois representation of $\Gal$ attached to the newform $f(\tau)\otimes \left( \frac{-1}{\cdot}\right)$ by the work of Deligne \cite{D}. Hence, 
\begin{equation}\label{eq:3}
\textrm{trace}(\rho_\ell'(\textrm{\Frob}_p))=\left( \frac{-1}{p}\right)\gamma(p) \textrm{  and  } \det(\rho_\ell'(\Frob_p))=\left(\frac{-24}{p}\right)p^2,
\end{equation}
for prime $p \ne 2,3$ and $\ell$.

We will prove that representations $\rho_\ell'$ and $\tilde{\rho_\ell}$ are isomorphic by showing that both of them are isomorphic to the representation $\rho_\ell$ which we define now.
Proposition \ref{prop:2} then follows from \eqref{eq:2} and \eqref{eq:3}.

Let $X(\Gamma_2)^{0}$ be the complement in $X(\Gamma_2)$ of the cusps. Denote by $i$ the inclusion of $X(\Gamma_2)^0$ into $X(\Gamma_2)$, and by $h':\mathcal{S} \rightarrow X(\Gamma_2)^0$ the restriction of elliptic surface $h:\mathcal{S}\rightarrow X(\Gamma_2)$ to $X(\Gamma_2)^0$. For a prime $\ell$ we obtain a sheaf
$$\mathcal{F}_\ell = R^1h'_{*}\Q_\ell$$
on $X(\Gamma_2)^0$, and also sheaf $i_{*} \mathcal{F}_\ell$ on $X(\Gamma_2)$ (here $\Q_\ell$ is the constant sheaf on the elliptic surface $\mathcal{S}$, and $R^1$ is derived functor). The action of $\Gal$ on the $\Q_\ell$-vector space
$$W = H_{et}^1(X(\Gamma_2)\otimes \overline{\Q}, i_* \mathcal{F}_\ell)$$
defines $\ell$-adic representation $\rho_\ell$. Representation is unramified outside $2,3$ and $\ell$.
By the argument similar to \cite[Proposition 5.1.]{LLY}, $\rho_\ell$ is isomorphic to $\tilde\rho_\ell$ up to a twist by quadratic character.

Using explicit equation for $\mathcal{S}$, we can calculate $trace(\rho_l(\Frob_p))$ and $$\det(\rho_l(\Frob_p))=\frac{1}{2}((trace(\rho_l(\Frob_p))^2-trace(\rho_l(\Frob_p^2)))=\frac{1}{2}((trace(\rho_l(\Frob_p))^2-trace(\rho_l(\Frob_{p^2})))$$ for $p \ne 2,3,\ell$ using the following theorem.

\begin{theorem}
	\label{thm:trace}
	Let $q=p^s$ be a power of prime $p\ne 2,3, \ell$. The following are true:
	\begin{itemize}
		\item[(1)] We have that
		\[
		\trace(\Frob_q|W)=-\sum_{t\in X(\Gamma_j)(\mathbb{F}_q)} \trace(\Frob_q|(i_*\mathcal{F}_\ell)_t).
		\]
		\item[(2)] If the fiber $E_t := h^{-1}(t)$ is smooth, then
		\[
		\trace(\Frob_q|(i_*\mathcal{F}_\ell)_t)=\trace(\Frob_q|H^1(E_t,
		\Q_\ell))=q+1-\#E_t(\mathbb{F}_q).
		\]
		
		\item[(3)] If the fiber $E^j_t$ is singular, then
		\begin{equation*}
			\trace(\Frob_q|(i_*\mathcal{F}_\ell)_t)=
			\begin{cases}
				1 & \text{if the fiber is split multiplicative}, \\
				-1 & \text{if the fiber is nonsplit multiplicative},\\
				0 & \text{if the fiber is additive}.
			\end{cases}		
		\end{equation*}	
	\end{itemize}
\end{theorem}		

\section{Serre-Faltings method and proof of Proposition \ref{prop:2}}\label{sec:Serre}

We will prove the following proposition.

\begin{proposition}\label{prop:3}
For every prime $\ell$ the representations $\rho_\ell$ and $\rho_\ell'$ are isomorphic.
\end{proposition}

Since the families are strictly compatible, by Chebotarev density theorem it is enough to prove that $\rho_2$ and $\rho_2'$ are isomorphic. We apply the method of Serre and Faltings as formulated in \cite[Section 5]{Scholl2}.

\begin{theorem}\label{thm:3}
For a finite set of primes $S$ of $\Q$, let $\chi_1, \ldots, \chi_r$ be a maximal independent set of quadratic characters of $\Gal$ unramified outside $S$, and $G$ a subset of $\Gal$ such that the map $(\chi_1, \ldots, \chi_r):G \rightarrow (\Z/2\Z)^r$ is surjective.

Let $\sigma, \sigma':\Gal \rightarrow \GL_2(\Q_2)$ be continuous semisimple representation unramified away from $S$, whose images are pro-2-groups. If for every $g\in G$
$$\trace(\sigma(g))=\trace(\sigma'(g)) \textrm{ and } \det(\sigma(g))=\det(\sigma'(g)),$$
then $\sigma$ and $\sigma'$ are isomorphic.
\end{theorem}

\begin{proposition}
Images of representations $\rho_2$ and $\rho_2'$ are pro-2-groups.
\end{proposition}
\begin{proof}
We can assume that the images of both representations are contained in $\GL_2(\Z_2)$. It is enough to prove that the images of their mod $2$ reductions have order $2$ (since the kernel of the natural homomorphism $\GL_2(\Z/2^{k+1}\Z)  \rightarrow \GL_2(\Z/2^{k}\Z)$ is a $2$-group).
For primes $p\in \{7,11,13\}$ using Theorem \ref{thm:trace} and an explicit model for surface $\mathcal{S}$, we compute that
\begin{equation*}
\textrm{trace}(\rho_2(\textrm{\Frob}_p))=\left( \frac{-1}{p}\right)\gamma(p) \textrm{  and  } \det(\rho_2(\Frob_p))=\left(\frac{-24}{p}\right)p^2.
\end{equation*}

Moreover, if $\left(\frac{-6}{p}\right)=-1$, we find that $\gamma(p)=0$ and the eigenvalues of $\rho_2(\Frob_p)$ are $\pm p\sqrt{-1}$ from which it follows that mod $2$ reduction of $\rho_2(\Frob_p)$ has order $2$. If $\left(\frac{-6}{p}\right)=1$, then the eigenvalues mod $2$ are equal, and mod $2$  reduction of $\rho_2(\Frob_p)$ is trivial.

Since the group $\GL_2(\Z/2\Z)$ is isomorphic to the symmetric group $S_3$, if we assume that the mod $2$ image is not of order two, then it must be the whole group. In that case, denote by $L$ a $S_3$ Galois extension of $\Q$ cut out by mod $2$ reduction of $\rho_2$ (i.e. $L$ is the fixed field of the kernel of the mod $2$ reduction of $\rho_2$). Then $L$ contains a unique quadratic field $K$ which is unramified outside $2$ and $3$ in which $7$ and $11$ split and $13$ is inert. It follows that $K=\Q(\sqrt{-6})$. We know by the Hermite-Minkowski theorem that there are finitely many $S_3$ extensions of $\Q$ unramified outside $2$ and $3$, and using LMFDB \cite{lmfdb} we find that there is only one such field $\Q(x)$, where $x^6-3x^2+6=0$, whose Galois group contains $\Q(\sqrt{-6})$. This field contains a cubic field $F = \Q(s)$, where $s^3+3s-2=0$. One finds that $7$ is inert in $F$, hence $\rho_2(\Frob_7)$ has order $3$. This is impossible since $\trace(\rho_2(\Frob_7))=10$ is an even number which implies that mod $2$ reduction of $\rho_2(\Frob_7)$ has order $1$ or $2$.
\end{proof}

To apply Theorem \ref{thm:3} for $S=\{2,3\}$ we choose characters 
$$\chi_1=\left( \frac{-1}{p}\right),  \chi_2=\left( \frac{2}{p}\right), \chi_3=\left( \frac{3}{p}\right),$$
and $G=\{\Frob_p:31 \le p \le 73, \textrm{ for } p \textrm{ prime}\}$.
Using Theorem \ref{thm:trace} and \eqref{eq:3} we can check that
$$\trace(\rho_2(g))=\trace(\rho'_2(g)) \textrm{ and } \det(\rho_2(g))=\det(\rho'_2(g)),$$
for all $g \in G$, hence Proposition \ref{prop:3} follows.

To prove Proposition \ref{prop:2} (and consequently Theorem \ref{thm:1}), we need to show that representations $\rho_\ell$ and $\tilde{\rho_\ell}$ are isomorphic. In particular, it is enough to prove this claim for $\ell=2$. By the argument similar to \cite[Proposition 5.1.]{LLY}, it follows that $\rho_2$ is isomorphic to $\tilde\rho_2$ up to a twist by a quadratic character. Since both representations are unramified outside $2$ and $3$, this character is an element of the group generated by characters $\chi_1$, $\chi_2$ and $\chi_3$. For every nontrivial $\chi$ from that group, we can find a prime $p>3$ such that $\chi(p)=-1$, and numerically check that ASD congruence relation for the Fourier coefficients of $g(\tau)$
$$c_{mp^r}-\chi(p)\left(\frac{-1}{p} \right)\gamma(p) c_{mp^{r-1}}+\left( \frac{-6}{p}\right)p^2 c_{mp^{r-2}}\equiv 0 \pmod{p^{2r}},$$
does not hold for some choice of $m$ and $r$. The claim follows.

All the computations in this paper were done in SageMath \cite{sagemath} and Magma \cite{magma}.

\section{Future work}

It is natural to ask do similar mod $p$ congruences exist for the numbers $F(\frac{p-1}{n})$, where $n>2$ and $p \equiv 1 \pmod{n}$? E.g. when $n=3$, by considering the $3$-cover (defined by $t=s^3$) of the elliptic surface $\mathcal{W}$, one can show that for $p\equiv 1 \pmod{3}$ we have $F\left(\frac{p-1}{3}\right) \equiv A_p \pmod{p}$, where $A_p$ is the trace of $\Frob_p$ under the Galois representation defined analogously to $\rho_\ell$ (in this situation the representation is four-dimensional).

In the paper under the preparation, we are going to investigate this phenomena for sequence $F$ and other Ap\'ery numbers.

\section{Acknowledgments}
The author would like to thank Robert Osburn and Armin Straub for bringing this problem to his attention.

The author was supported by the QuantiXLie Centre of Excellence, a project co-financed by the Croatian Government and European Union through the European Regional Development Fund - the Competitiveness and Cohesion Operational Programme (Grant KK.01.1.1.01.0004), and by the Croatian Science Foundation under the project no. IP-2018-01-1313.


\bibliographystyle{siam}
\bibliography{bibl}

\begin{thebibliography}{10}

\bibitem{Ap}
{\sc R.~Ap\'ery}, {\em Irrationalit\'e de {$\zeta 2$} et {$\zeta 3$}},
  Ast\'erisque,  (1979), pp.~11--13.
\newblock Luminy Conference on Arithmetic.

\bibitem{ASD}
{\sc A.~O.~L. Atkin and H.~P.~F. Swinnerton-Dyer}, {\em Modular forms on
  noncongruence subgroups},  (1971), pp.~1--25.

\bibitem{B1}
{\sc F.~Beukers}, {\em Another congruence for the {A}p\'ery numbers}, J. Number
  Theory, 25 (1987), pp.~201--210.

\bibitem{magma}
{\sc W.~Bosma, J.~Cannon, and C.~Playoust}, {\em The {M}agma algebra system.
  {I}. {T}he user language}, J. Symbolic Comput., 24 (1997), pp.~235--265.
\newblock Computational algebra and number theory (London, 1993).

\bibitem{D}
{\sc P.~Deligne}, {\em Formes modulaires et repr\'esentations de {${\rm
  GL}(2)$}},  (1973), pp.~55--105. Lecture Notes in Math., Vol. 349.

\bibitem{K}
{\sc M.~Kazalicki}, {\em Modular forms, hypergeometric functions and
  congruences}, Ramanujan J., 34 (2014), pp.~1--9.

\bibitem{KS}
{\sc M.~Kazalicki and A.~J. Scholl}, {\em Modular forms, de {R}ham cohomology
  and congruences}, Trans. Amer. Math. Soc., 368 (2016), pp.~7097--7117.

\bibitem{LL}
{\sc W.-C.~W. Li, L.~Long, and Z.~Yang}, {\em Modular forms for noncongruence
  subgroups}, Q. J. Pure Appl. Math., 1 (2005), pp.~205--221.

\bibitem{LLY}
\leavevmode\vrule height 2pt depth -1.6pt width 23pt, {\em On
  {A}tkin-{S}winnerton-{D}yer congruence relations}, J. Number Theory, 113
  (2005), pp.~117--148.

\bibitem{lmfdb}
{\sc T.~{LMFDB Collaboration}}, {\em The l-functions and modular forms
  database}, 2018.

\bibitem{OS}
{\sc R.~Osburn and A.~Straub}, {\em Interpolated sequences and critical
  l-values of modular forms}, arXiv:1806.05207 [math.NT].

\bibitem{Scholl}
{\sc A.~J. Scholl}, {\em Modular forms and de {R}ham cohomology;
  {A}tkin-{S}winnerton-{D}yer congruences}, Invent. Math., 79 (1985),
  pp.~49--77.

\bibitem{Scholl2}
\leavevmode\vrule height 2pt depth -1.6pt width 23pt, {\em The {$l$}-adic
  representations attached to a certain noncongruence subgroup}, J. Reine
  Angew. Math., 392 (1988), pp.~1--15.

\bibitem{Shi}
{\sc G.~Shimura}, {\em Introduction to the arithmetic theory of automorphic
  functions}, Publications of the Mathematical Society of Japan, No. 11.
  Iwanami Shoten, Publishers, Tokyo; Princeton University Press, Princeton,
  N.J., 1971.
\newblock Kan\^o Memorial Lectures, No. 1.

\bibitem{SB}
{\sc J.~Stienstra and F.~Beukers}, {\em On the {P}icard-{F}uchs equation and
  the formal {B}rauer group of certain elliptic {$K3$}-surfaces}, Math. Ann.,
  271 (1985), pp.~269--304.

\bibitem{sagemath}
{\sc {The Sage Developers}}, {\em {S}ageMath, the {S}age {M}athematics
  {S}oftware {S}ystem ({V}ersion 7.2)}, 2018.
\newblock {\tt http://www.sagemath.org}.

\bibitem{V}
{\sc H.~A. Verrill}, {\em Congruences related to modular forms}, Int. J. Number
  Theory, 6 (2010), pp.~1367--1390.

\bibitem{Zag1}
{\sc D.~Zagier}, {\em Integral solutions of {A}p\'ery-like recurrence
  equations}, in Groups and symmetries, vol.~47 of CRM Proc. Lecture Notes,
  Amer. Math. Soc., Providence, RI, 2009, pp.~349--366.

\bibitem{Zag2}
\leavevmode\vrule height 2pt depth -1.6pt width 23pt, {\em Arithmetic and
  topology of differential equations}, in Proceedings of the 2016 ECM, 2017.

\end{thebibliography}

\end{document}